\documentclass[leqno]{amsart}
\usepackage{amssymb,latexsym,amsmath,amscd,epsfig,amsthm, amsfonts}

\newtheorem{lemma}{Lemma}
\newtheorem{theorem}{Theorem}
\newtheorem*{corollary}{Corollary}

\newtheorem*{facta}{Theorem A}

\newtheorem*{factla}{Lemma A}
\newtheorem*{factlb}{Lemma B}

\newtheorem{proposition}{Proposition}

\theoremstyle{remark}

\newtheorem*{remark}{Remark}

\newcommand{\B}{{\mathbf B}}
\newcommand{\N}{{\mathbf N}}
\newcommand{\R}{{\mathbf R}}

\newcommand{\conv}{{\rm conv}}
\newcommand{\diam}{{\rm diam}}

\newcommand{\beq}{\begin{equation}}
\newcommand{\eeq}{\end{equation}}
\newcommand{\beqs}{\begin{equation*}}
\newcommand{\eeqs}{\end{equation*}}
\newcommand{\ba}{\begin{eqnarray}}
\newcommand{\ea}{\end{eqnarray}}
\newcommand{\bas}{\begin{eqnarray*}}
\newcommand{\eas}{\end{eqnarray*}}

\newcommand{\Bm}{\mathbf B}
\newcommand{\Sm}{S}

\newcommand{\HH}{{\mathcal H}}

\newcommand{\z}{\zeta}
\newcommand{\p}{\phi}
\newcommand{\vp}{\varphi}

\newcommand{\fraction}{\frac}

\begin{document}
\author{Kjersti Solberg Eikrem and Eugenia Malinnikova}
\title
{Radial growth of harmonic functions in the unit ball} 
\address{Department of Mathematical Sciences, Norwegian University of Science and
Technology, NO-7491, Trondheim, Norway}
\email{kjerstei@math.ntnu.no}
\email{eugenia@math.ntnu.no}

\begin{abstract}
Let $\Psi_v$ be the class of harmonic functions in the unit disk or unit ball in $\R^n$ which admit a radial majorant $v(r)$. We prove that when $v$ fulfills a doubling condition, a function in $\Psi_v$ may grow or decay as fast as $v$ only along small sets of radii, and we give precise estimates of these exceptional sets in terms of Hausdorff measures. 
\end{abstract}

\thanks{The second author was supported by the Research
Council of Norway, grant 185359/V30}

\maketitle

\section{Introduction}
Radial behavior of harmonic functions in the unit disk and unit ball in $\R^m$ is a classical topic in analysis. In this article we consider harmonic functions bounded a priori by some radial majorant and discuss their radial growth. 

It follows from a theorem of N. N. Lusin and I. I. Privalov, see \cite{P}, that there exist harmonic functions in the unit disk that tend to infinity along almost each radius. Moreover, a generalization of this result obtained by J.-P. Kahane and Y.~Katznelson \cite{KK}, shows that such functions may be bounded by an arbitrarily slow growing radial majorant.

Let $v(r)$ be a positive increasing continuous function on $[0,1)$ and assume that 
 $\lim_{r\rightarrow 1} v (r) =+\infty $. Let $\B$ be the unit ball in $\R^m$, we define
 \beq \label{1}
 \Phi^m_v=\{u:\B\rightarrow\R, \Delta u=0, u(x)\le Kv(|x|)\},\eeq
and
 \[
 \Psi^m_v=\{u:\B\rightarrow\R, \Delta u=0, |u(x)|\le Kv(|x|)\}.\]

Harmonic functions of the class $\Phi^2_v$ with $v(r)=|\log(1-r)|$ were studied by B. ~Korenblum in \cite{K}. This class as well as more general classes that correspond to $v(r)=|\log(1-r)|^s$ appear in connection with the related 
spaces of analytic functions, see also \cite{Se,BL}.
Radial growth of harmonic functions in the unit disk bounded by a multiple of $|\log(1-r)|$ was studied in \cite{BLMT} and \cite{LM}. The aim of this article is to understand to what extent some of the results in \cite{BLMT} 
remain true for general majorants and  higher dimensional spaces.

We mostly consider functions $v$ that satisfy the following doubling condition
\begin{equation}\label{vdouble}
v(1-d/2)\le D v(1-d).
\end{equation} 
The constants $K $ and $D $ will preserve their identities throughout this article.
 
The main aim of this work is to estimate the size of the set of the radii along which a function from $\Phi^m_v$ or $\Psi^m_v$ grows or decays as fast as the majorant $v(r)$. 
 For each function $u\in\Phi^m_v $ we define subsets of the unit sphere 
 \begin{gather*}
 E^-(u)=\{y\in S: \limsup_{r\rightarrow 1-}\frac{u(ry)}{v(r)}<0\},\\
 E^+(u)=\{y\in S: \liminf_{r\rightarrow 1-}\frac{u(ry)}{v(r)}>0\}.
 \end{gather*}
For every increasing continuous function $\lambda:[0,+\infty)\rightarrow [0,+\infty)$ with $\lambda(0)=0$ we denote by $\mathcal{H}_\lambda$ the corresponding Hausdorff measure.
\begin{theorem}
\label{th:haus}
Let $v$ satisfy (2).

\noindent(a) If $u\in\Phi_v^m$ and $\lambda$ is a continuous increasing function, $\lambda(0)=0$ and 
$$\lambda(t)=o(t^{m-1}v(1-t)^\alpha), \qquad(t\rightarrow 0),$$ for any $\alpha>0$,
then $\HH_\lambda(E^+(u))=\HH_\lambda(E^-(u))=0$.\\
\noindent(b) For any $\beta>0$ there exists $u\in\Psi_v^m$ and an increasing continuous function $\lambda_\beta$, $\lambda_\beta(0)=0$ and $\lambda_\beta(t)=O(t^{m-1}v(1-t)^\beta)\ (t\rightarrow 0)$, such that $\HH_{\lambda_\beta}(E^+(u))>0 $. 
\end{theorem}
Since $u\in \Psi_v^m$ and $E^-(-u) =E^+ (u) $, the estimate for $E^- $ in (a) is also sharp. 
 
In this theorem there is no difference between the size of the sets $E^\pm(u)$ for $u\in\Phi_v^m$ and $u\in\Psi_v^m$. The situation is different for positive harmonic functions as was also noted in \cite{BLMT}. We generalize the result on positive harmonic functions to a wide class of weights and show also that no a priori growth estimate is needed. More precisely, we obtain the following:

\begin{theorem} 
\label{th:pi}
Assume that $\lambda(t)=t^{m-1}v(1-t)$ is an increasing continuous function and $\lambda(0)=0$. \\
\noindent (a) For any positive harmonic function $u$ in the unit ball of $\R^m$ we define
\beq
\label{eq:F}
F_v^+(u)=\{y\in S:\limsup_{r\rightarrow 1}\frac{u(ry)}{v(r)}>0\}.
\eeq
Then $F_v^+(u)$ is the countable union of sets of finite $\HH_\lambda$-measure.\\
\noindent (b) There exists a positive function $u\in\Psi^m_v$ such that $\HH_\lambda(E^+(u))>0$.
\end{theorem} 

The article is organized as follows. 
We collect some preliminary results on harmonic measure and Hausdorff measures in the next section. Then we prove Theorem \ref{th:haus}. For part (a) our arguments are similar to those in \cite{BLMT}, but in higher dimensions they are based on estimates of harmonic measure due to B.~E.~ Dahlberg, \cite{D}. A new approach is used to construct examples of functions with a large set of extremal growth in dimension larger than two in the proof of Theorem \ref{th:haus} (b). 

Finally, in the last section we study the radial growth of positive harmonic functions. We prove Theorem \ref{th:pi} and describe boundary measures that correspond to positive functions in $\Psi^m_v$.


\section{Preliminaries}
\subsection{Poisson kernel and some estimates}

Let $\sigma $ be the $(m - 1) $-dimensional 
surface measure on $S $ and denote $\sigma (S) =\gamma_{m -1} $.
The Poisson kernel in the  $m $-dimensional unit ball is
$$P(x,\z) =\frac1{\gamma_{m -1}}\frac{1 - |x |^2}{|x -\z |^m}. $$

 Assume for simplicity that $x = (1,0,...,0) $. 
 Using hyperspherical coordinates for $\z\in S $ we have
 $\z= (\cos\phi,\z') $, where $|\z' | =\sin\phi $.
Let 
$$ \tilde{P}_{m,r} (\phi) =\fraction1{\gamma_{m -1}}\fraction{1 -r^2}{(1 +r^2 -2r\cos \phi)^{m/2}} $$
and
  \[
Q_m(r,\phi)=-\partial_{\phi} \tilde{P}_{m,r} (\phi)=
       \frac{1}{\gamma_{m -1}}\frac{mr(1-r^2)\sin\phi}{(1 +r^2 -2r\cos \phi)^{(m +2)/2}}. 
       \]
Then $P(rx,\z)=\tilde{P}_{m,r} (\phi)$.

Let $d (x,\z) $ be the geodesic distance between two points $x $ and $\z $ on $S $. Then let $B (x,\p)= \{\z\in S: 
 d(x,\z) < \p\} $ be the hyperspherical cap of radius $\p $ with center in $x $. It can be shown that for the $(m -1) $-dimensional surface measure of the cap 
 \beq
 \label{cap}
  C_1 \p^{m -1}\le \sigma ( B (x,\p) ) \le C_2 \p^{m -1},
  \eeq 
  where the constants depend on $m $.

We will  need some estimates for integrals of $Q_m $. 
\medskip

\noindent {(i)}
 We have
\begin{gather*}
\int_{0}^{1 - r} Q_m(r,\phi)d\phi 
\le C_3 \int_{0}^{1 - r} \frac{r(1-r^2)\phi}{((1 - r)^2 +2r (1 -\cos \phi) )^{(m +2)/2}} d\phi \\
\le C_3\int_{0}^{1 - r} \frac{r(1-r^2)\phi}{(1 - r)^{m +2}} d\phi, 
\end{gather*} 
hence
\beq\label{q1}
\int_{0}^{1 - r} Q_m(r,\phi)d\phi \le C_4\fraction{1}{(1 - r)^{m -1} }.
\eeq
\medskip

\noindent {(ii)}
 For $d>0 $ 
\begin{gather*}
\int_{d}^{\pi} Q_m(r,\phi)d\phi 
\le C_3 \int_{d}^{\pi} \frac{r(1-r^2)\phi}{((1 - r)^2 +r\fraction{4}{\pi^2}\phi^2)^{(m +2)/2}} d\phi 
\le C_3 \int_{d}^{\pi} \frac{r(1-r^2)\phi}{(r\fraction{4}{\pi^2}\phi^2)^{(m +2)/2}} d\phi,
\end{gather*} 
thus
\beq
\label{q2}
 \int_{d}^{\pi} Q_m(r,\phi)d\phi \le C_5 r^{-m/2} d^{-m}\le C_6  d^{-m}
\eeq 
when $r> \fraction12$.

\medskip

\noindent {(iii)}
 Furthermore, by (\ref{cap}),
\begin{gather*}
 \int_{1 - r}^\pi \sigma (B (x,\phi))Q_m(r,\phi)d\phi\le \int_{1 -r}^\pi C_2 \phi^{m-1}  Q_m(r,\phi)d\phi \\
 \le C_7\int_{1 -r}^\pi\fraction{ (1 -r^2)\phi^{-2}} {r^{m/2 }} d\phi \le C_8 r^{-m/2},
\end{gather*} 
so for $r> \fraction12 $,
\beq
\label{q3}
 \int_{1 - r}^\pi \sigma (B (x,\phi))Q_m(r,\phi)d\phi \le C_9 .
\eeq

\subsection{Harmonic measure in Lipschitz domains} 
A bounded domain $\Omega\in \R^m $ is a Lipschitz domain if there is a constant $C $ such that to each point $q\in \partial \Omega $ there corresponds a coordinate system $(\xi,\eta) , \xi\in\R^{m -1},\eta\in\R$, and a function $\vp $ such that $|\vp (\xi_1) -\vp (\xi_2) |\le C|\eta_1 -\eta_2 | $ for  
$D\cap V =\{(\xi,\eta):\vp(\xi)<\eta\} $ for some neighborhood $V $ of $q $. The smallest such constant is called the Lipschitz constant.

Let $S $ be the 
 unit sphere in $\R^m $.
For $\z\in S$ and $a<1$ we use the standard notation $\Gamma^a_\z=\conv(\z, a\B)$ for the convex hull of $\z$ and the $m $-dimensional ball of radius $a$. 
Given a compact set $F\in \Sm $ we consider the cone-domain $G =G (F,a) =\cup_{\z\in F} \Gamma_\z^a$. It is a Lipschitz domain, and the Lipschitz constant of $G (F,a) $ depends on $a $ only.
Given a Jordan domain $\Omega$, a subset $A\subset\partial\Omega$ and a point $z\in \Omega$, we denote by $\omega(z,A,\Omega)$ the harmonic measure of $A$ at point $z$. 

A celebrated result by B. E. Dahlberg \cite{D} says that on the boundary of a Lipschitz domain the harmonic measure and the surface measure are mutually absolutely continuous. We need a quantitative form of this result for cone-domains and refer the reader to \cite{D, JK} and \cite[Chapter 4.2]{BM}.

\begin{facta} 
Let $a>0$, then there exist $\alpha$ and $ C$ that depend on $a$ and $m$ only such that
for any cone-domain $G=G(F,a)$ in the unit ball of $\R^m$ and any $A\subset Q\subset\partial G$
the following inequality holds
\[
\frac{\omega(0,A,G)}{\omega(0,Q,G)}\ge C\left(\frac{\eta(A)}{\eta(Q)}\right)^\alpha,
\]
where $Q$ is a ball on $\partial G$ and $\eta$ is the surface measure on $\partial G$.  
\end{facta}

\subsection{Hausdorff measures}
\label{preliminaries haus}

We will refer to generalized Hausdorff measures in $\R^m$ as they are defined for example in \cite[p.\,59]{M}. Let $h$ be an increasing continuous function on $[0,+\infty), h(0)=0$, then for any $E\subset\R^m$
\[\HH_h(E)=\liminf_{\delta\rightarrow 0}\{\sum_j h(d_j): E\subset\cup_jF_j, d_j=\diam(F_j)<\delta\}.\]
We assume in addition that $h(t/2)\ge ch(t)$ for some $c>0$. Then the Hausdorff measure is equivalent to the  
so-called net measure $N_h(E)$ defined with $F_j$ being half-open dyadic cubes with sides parallel to the coordinate axis in the following sense: $\HH_h(E)\le N_h(E)\le A(c,m)\HH_h(E)$, see \cite[p.\,76]{M}. Further, the following property holds, if $f:\R^k\rightarrow \R^m$ is a Lipschitz map and $E\subset\R^k$, then $\HH_h(f(E))\le L\HH_h(E)$, where $L$ depends on the Lipschitz constant of $f$ and on $c$. The proofs follow readily from the definitions.

We will use  Cantor-type  sets  having the following structure:
\[
C=\cap_s C_s,  \ C_s \supset C_{s+1}, \ C_0=[0,1], 
\]
each set $C_s$ is a union of $N_s$  segments $\{I^{(s)}_j\}_j$ of the same length $l_s$. For each such segment the intersection $C_{s+1}\cap I^{(s)}_j$
is a union of $k_s$  non-overlapping segments of length $l_{s+1}$. We assume, of course, that
\beqs
\label{eq:34}
(i)\ l_s \searrow 0 \ \text{as} \ s \to \infty, \quad  (ii)\ k_sl_{s+1} < l_s, \quad \text{and}  \quad  (iii) \ N_s=k_0k_1\ldots k_{s-1}.
\eeqs
The next result is Theorem 3 in \cite{BLMT}.

\begin{factla}Let   $\lambda:[0,1)\rightarrow[0,+\infty)$ be a continuous increasing function with $\lambda(0)=0$, such that  for some $a>0$ and $s>s_0$
\begin{equation}
\label{eq:cantorh} 
\frac{\lambda(l)}{l}\ge a\frac{\lambda(l_{s+1})}{l_{s+1}}\quad{\rm{
for\ any}}\ l\in[l_{s+1},l_s).
\end{equation}
Then
\beq
\label{eq07}
\liminf_{s\rightarrow\infty} N_s\lambda(l_s) 
                   \ge \HH_\lambda(C)
                           \ge \frac{a}{2}\liminf_{s\rightarrow\infty} N_s\lambda(l_s).
\eeq
\end{factla}

Two slightly more delicate results that we need, give estimates of the Hausdorff measure of (symmetric) Cantor sets and cylinder sets in higher dimensions. Note also that we are not interested in the exact value of the Hausdorff measure but only in its positivity.

\begin{factlb}[Hatano, \cite{HA}] 
Let $\{k_q\}_{q=1}^\infty$ be a sequence of positive integers and $\{l_q\}_{q=0}^\infty$, $l_0=1$ be a sequence of positive numbers that satisfy
$k_{q+1}l_{q+1}<l_q$. The generalized symmetric Cantor set $E$ in $\R^m$ defined by the sequences $\{k_q\}$ and $\{l_q\}$ is constructed in the following way: Let $C_0=[0,1]$, $C_1$ is obtained from $C_0$ by removing $k_1-1$ open intervals of equal lengths such that remaining $k_1$ closed intervals are of length $l_1$. Then, to get $C_2$, $k_2-1$ open intervals are removed from each interval of $C_1$ such that remaining intervals are of length $l_2$, etc.  Define $C=\cap_n C_n$ and $E=C\times C\times...\times C$.

Then $\HH_h(E)>0$ if and only if  $\liminf_{q\rightarrow\infty}(k_1...k_q)^mh(l_q)>0.$
\end{factlb}

The measure used in \cite{HA} is not the classical Hausdorff measure but one defined using coverings by all open cubes. As we mentioned above, under 
our condition on $h$ the two measures are equivalent (up to a multiplicative constant). 

The next statement is intuitively clear but we were not able to find a precise reference, so we outline a short proof.

\begin{lemma}
\label{crossproduct} Let $h(t)=t^{k-1}\nu(t)$, where $\nu$ is an increasing continuous function on $[0,+\infty)$ and $\nu(0)=0$. Assume also that $\nu(t/2)\ge c\nu(t)$ for some $c>0$. If $F\subset [0,1]$ is compact, $\HH_\nu(F)>0$, and $E=F\times [0,1]^{k-1}\subset\R^k$, then $\HH_h(E)>0$. 
\end{lemma}

\begin{proof}
We will use that $\HH_\nu$ is equivalent to $N_\nu$ and $\HH_h$ is equivalent to $N_h$. Assume that $N_h(E)=0$, then for any $\epsilon>0$ and $\delta>0$ there exists a \textit{finite } family of half-open dyadic cubes $\{Q_\alpha\}$ with sides $l_\alpha=2^{-n_\alpha}<\delta$ that covers $E=F\times [0,1]^{k-1}$ and such that $\sum_\alpha h(l_\alpha)<\epsilon$.  Indeed we can find an infinite family for which $\sum h(l_\alpha)<2^{-k}\epsilon$, then for each cube $Q$ in this family, take an open cube that contains $Q$, has side length which is twice that of $Q$ and can be covered by $2^k$ half-open dyadic cubes of the same size as $Q$. Then we choose a finite sub-cover of the compact set $E$.

Let $n=\min_{\alpha} n_\alpha$ and $N=\max_{\alpha} n_\alpha$, we divide $[0,1]$ into dyadic intervals of length $2^{-n}$, $[0,1]=\cup_j I_j$. For each $j$ consider cubes $K_{j,s}=I_j\times J_s, 1\le s\le 2^{n (k -1)}$, where $J_s $ is a dyadic subcube of $[0,1]^{k -1} $  with side length $2^{-n} $. 
 Now for each $s$ let
\[d_{j,s}=\sum_{\alpha: Q_\alpha\subset K_{j,s}} h(l_\alpha).\]
Choose $t=t(j)$ such that $d_{j,t}=\min_s d_{j,s}$ and
replace the covering $\{Q_\alpha\}$ by a new one $\{Q_\beta\}$ such that $\sum_\beta h(l_\beta)\le \sum_\alpha h(l_\alpha)$, and for each $j$ the cubes $\{Q_\beta\}$ contained in $K_{j,s}$ can be obtained from the cubes $\{Q_\alpha\}$ contained in $K_{j,t(j)}$ by 
translation. If for the new family $\min n_\beta>n$, we repeat the procedure. If not, we get some chains of cubes $K_{j,1},...,K_{j,2^{(k-1)n}}$ in the new family and repeat the procedure on the complement of these chains. Anyway the size of the smallest cubes is always at least $2^{-N}$ and after finitely many steps we find  a family of intervals $I_\gamma$ of length $l_\gamma<\delta$ that covers $F$ and 

\[\sum_\gamma \nu(l_\gamma)=\sum_\gamma h(l_\gamma)l_\gamma^{-(k-1)}<\epsilon.\] Thus $\HH_\nu(F)=0$. 
\end{proof}

\section{Sets of extremal growth or decay}

\subsection{Lebesgue measure of sets of extremal growth}
In this subsection we first estimate the Lebesgue measure of the sets $E^\pm(u)$.
\begin{proposition}
Suppose that $u\in\Phi^m_v$, then $\sigma (E^-(u))=0$.
\end{proposition}

\begin{proof}
We have $E^{-}(u)=\cup_nF_n=\cup_n\{\z \in S: u(r\z)\le-\frac1n v(r), r\ge 1-\frac1n\}.$
Assume $\sigma(E^-(u))>0$. Then  $\sigma (F_n)>0$ for some $n$, and $F_n$ is a compact subset of $S$.  Let $G=\cup_{\z\in F_n} \Gamma^a_\z$ and $G_\alpha=G\cap \alpha\B$ for $\alpha<1$. 
We have $\partial G=F_n\cup L$, where $L=\partial G\cap \B$.

We will estimate $u(0)$ using harmonic measure in domain $G_\alpha$.
First, it follows from Dahlberg's theorem 
that $\omega(0 ,F_n, G)=c>0$. Now let $L_\alpha=\partial G\cap \alpha\B=L\cap \alpha\B$ and let $p_\alpha(A)$ be the radial projection of a set $A$ onto $\alpha S $, where $0<\alpha\le 1$. Then
\[
\partial G_\alpha=L_\alpha\cup \alpha F_n\cup p_\alpha(L\setminus L_\alpha).\]
Choose $s>1-\frac1n$ such that 
\[\omega(0, L\setminus L_s, G)<\frac{c}{3n}\quad{\rm\ and\ }\quad \sigma (p_1(L\setminus L_s))<\frac{\gamma_{m -1}c}{3n}.\]
Let $s <t <1 $. Then, since $G_t\subset G$, 
\[
\omega(0,L_t\setminus L_s,G_t)\le\omega(0,L_t\setminus L_s, G)\le\omega(0,L\setminus L_s,G)<\frac{c}{3n}.\]
Further, $$\omega(0, p_t(L\setminus L_t), G_t)\le\omega(0, p_t(L\setminus L_t), t\B)=\frac1{\gamma_{m -1}}\sigma (p_1(L\setminus L_t))<\frac{c}{3n}.  $$
Finally, 
we want to estimate $\omega(0, tF_n, G_t)$. Note that $tG\subset G_t$, 
then
\[
\omega(0, tF_n, G_t)\ge\omega(0,tF_n,tG)=\omega(0,F_n,G)=c.\] 

Now we apply the estimates for the function $u$, which is harmonic in $G_t$, using that $\partial G_t=L_s\cup(L_t\setminus L_s)\cup tF_n\cup p_t(L\setminus L_t).$
\[
u(0)\le v(s)\omega(0,L_s,G)+v(t)\frac{2c}{3n}-\frac{v(t)}{n}c\le v(s)-\frac{v(t)c}{3n}.\]
When $t$ goes to $1$ we get a contradiction, since $v(t)\rightarrow\infty$. 
\end{proof}

To deal with the set $E^+(u)$ we assume  that the function $v$ fulfills (\ref{vdouble}). The proof follows the argument from \cite{BLMT}.
\begin{lemma}\label{l:n-cone}
Let $u\in\Phi^m_v$ where $v$ satisfies (\ref{vdouble}) and assume 
$u(x)>c v(|x|)$ for some $x\in\B$. Then there exists $\tau=\tau(K, D, c)>0$ such that $u(x')>c/2 v(|x |)$ whenever $|x -x' | <\tau \,(1-|x |)$, $|x'| = |x | $.

The same statement holds if we write $<$ in both inequalities and assume that $c <0 $. 
\end{lemma}

\begin{proof} 
Let $y,y'\in\B$. Assume that $|y|=|y'|$ and $|y -y' | <\tau_1 \,(1-|y |)$, where $\tau_1 <1$, 
then for any $\z\in S$
$$|y ' -\z | \ge |y -\z | - | y -y' |> |y -\z | -\tau_1 |y -\z |= (1 -\tau_1) | y-\z |. $$
Thus
$|y -\z |^m (1 -\tau_1)^m <|y' -\z |^m  $ and
\beq
\label{poi}
P (y,\z)> (1 -\tau_1)^mP (y',\z).
\eeq

Let $r = |x | $, $R= (1+r)/2$ and denote $q=q(\tau_1)= (1-\tau_1)^m$.
We apply (\ref{poi}) with $y =\fraction{x}R$, $y' =\fraction{x'}R$ and $|y -y' | <\tau_1 \,(1-|y |)$. 
Then 
\bas 
u(x)&=&\int_S  u(R\z)P\left(\frac{x}{R},\z \right)d\sigma (\z) \\
&=&qu(x')+
\int_S u(R\z)\left(P\left(\frac{x}{R},\z\right)-
q P\left(\frac{x'}{R},\z\right)\right)d\sigma (\z)\\
& \le &qu(x')+
\int_S Kv(R)\left(P\left(\frac{x}{R},\z\right)-
q P\left(\frac{x'}{R},\z\right)\right)d\sigma (\z)\\
&=&qu(x')+(1 -q) Kv(R)\le qu(x')+(1 -q) KD v(r).
\eas 
If $\tau_1$ is such that $c-(1 -q)KD\ge \frac c 2q$ and $|x -x' | < \frac{\tau_1} 2 (1-r)<\tau_1 \left ( 1- \frac r R \right ) $, then $u(x') >\frac  c  2 v(r)$. To complete the proof it suffices to  choose $\tau(K,D,c)=\tau_1/2$.  

For the second case when $c<0$, we use the inequality
$$u (x ')\le qu (x)+(1 -q)KD v(r)<\left(qc +(1 -q) KD\right) v (r)$$
and choose $\tau_1 $ such that $q c+\left (1 - q\right) KD \le \frac{c}2 $. 

\end{proof}

\begin{corollary}
If $u\in\Phi^m_v $ where $v $ satisfies (\ref{vdouble}), then $\sigma (E^+ (u) )=0 $.
\end{corollary}

\begin{proof}
Note that by Lemma \ref{l:n-cone} $u$ is bounded from below in $\Gamma_\z^a$ for any $\z\in E^+(u)$ and some $a=a(\z)$. Then by results of L. Carleson \cite{Ca} (see also \cite{HW, D}), $u$ has finite non-tangential limit at almost each point of $E^+(u)$. Applying the lemma once again, we see that the non-tangential limit at $\z\in E^+(u)$ is infinite. Thus $\sigma (E^+ (u))=0$.  
\end{proof}


\subsection{Estimates of Hausdorff measures}
For weights that satisfy the doubling condition we can give more precise estimates of the size of exceptional set. We now prove Theorem \ref{th:haus} (a) formulated in the introduction. 

\begin{proof}[ Proof of Theorem  \ref{th:haus} (a)]
We start with $E^+(u)$. It is enough to prove the statement for each set  
\[E_n=\left\{\z \in S: u(r\z)\ge\frac1n v(r), r\ge 1-\frac1n\right\}.\]
By Lemma \ref{l:n-cone}, there exists $a$ such that 
$u(x)\ge\frac1{2n}v(|x |)$ for any $x\in\Gamma^a_\z$ where $|x|>1-\frac1n$ and  $\z \in E^+(u)$. 

Let $G=\cup_{\z\in E_n}\Gamma^a_\z$ and $G_t=G\cap t\B$. Clearly we may assume that $u\ge c_0$ on $G$ for some $c_0<0$. Let $b $ be such that 
\[
\partial G_t\cap tS=tE_n^{b(1-t)}=\{t\z \in tS: |\z-\z_0|<b(1-t) \ {\rm where}\ \z_0\in E_n\},\]
here $b=b(a)$.
Then by 
harmonic measure estimate for $G_t$ when $t>1-\frac1n$, we obtain
\[
u(0)\ge c_0+\omega(0, tE_n^{b(1-t)}, G_t) \frac{v(t)}{2n}.\]
By  Theorem A 
there exists $C$ and $\gamma>0$ that depend only on $a$ such that
\[\omega(0, A, G_t)\ge C\sigma (A)^\gamma,\] here $\sigma $ is the $(m - 1) $-dimensional surface 
measure on  $tS$.
This implies
\[
\sigma (E_n^{b(1-t)})^\gamma\le \frac{C_1}{v(t)}\]
where $C_1 =C_1 (n,u,a) $.
So for all $\epsilon>0$ small enough we get by applying (\ref{vdouble}) 
\[\sigma(E_n^\epsilon)\le C_2 \left(v\left(1-\frac{\epsilon}{b}\right)\right)^{-1/\gamma}\le C_3  \left(v(1-\epsilon)\right)^{-1/\gamma}.\]
We cover $E_n $ by a finite collection of balls $\{  B_j:{j\in J}\} $ of radius $\frac{\epsilon}{5} $ and centers at points in $E_n $. By the Vitali covering lemma (see for example \cite[p.\,2]{He}) there exists a subcollection $J' \subseteq J$ where $\{B_j:j \in J'\} $ are disjoint and 
$\cup_{j \in J} B_j \subseteq \cup_{j \in J'} 5B_j$, and we also have $\cup_{j \in J'} 5B_j  \subseteq E_n^\epsilon $. 
Then $E_n $ can be covered by $N_\epsilon $ balls $\{5B_j:j\in J'\} $ of radius $\epsilon $, where
$$  \epsilon^{m-1} N_\epsilon \le 5^{m-1} \sigma(E_n^\epsilon),$$
thus
$$N_\epsilon \le 5^{m-1} \epsilon^{-m+1} C_3 \left( v(1-\epsilon)\right)^{-1/\gamma}.  $$
Then 
$$ \mathcal{H}_\lambda (E_n) \le\liminf_{\epsilon \rightarrow 0} N_\epsilon \lambda (\epsilon) \le\liminf_{\epsilon \rightarrow 0}5^{m-1} \epsilon^{-m+1} C_3 \left( v(1-\epsilon)\right)^{-1/\gamma} \lambda (\epsilon).  $$
Since $\lambda(t)=o(t^{m -1}(v(1-t))^w), (t\rightarrow 0)$, for any $w>0 $, we get $\mathcal{H}_\lambda(E_n)=0.$
 
The proof for $\mathcal{H}_\lambda(E^-(u))$ is similar;  we then use the second part of Lemma \ref{l:n-cone}.
\end{proof}

\begin{remark}
If $g (x) =x^{\gamma} $ for $\gamma>0 $ and $u\in\Phi^m_v $, then the theorem above implies in particular that $\mathcal{H}_\lambda(E^+(u))=0$ and $\mathcal{H}_\lambda(E^-(u))=0$ when $\lambda (t) =t^{m -1}\log \fraction1t $. On the other hand, we will show in section \ref{construction} that for any $\epsilon >0 $ there exists $u\in\Phi^m_v $ such that $$ \mathrm{dim}\, E^+ (u)>m -1 -\epsilon. $$
\end{remark}

\subsection{Auxiliary functions} 

We now begin to prove Theorem \ref{th:haus} (b).  First we construct auxiliary functions $u_k $ in $B$ that resemble $\Im(z^{2^k})$ in the unit disk.

For each positive integer $k$ let $S_k$ and $T_k$ be subsets of the interval $[0,2\pi)$ defined by
\[S_k=\cup_{j=0}^{2^k-1}[2j\pi2^{-k},(2j+1)\pi2^{-k}),\]
\[T_k=\cup_{j=0}^{2^k-1}[(2j+1/4)\pi2^{-k}, (2j+3/4)\pi2^{-k}].\]
 Then on the unit sphere $\Sm$ in $\R^m$ we define
\[E_k=\{\eta\in\Sm, \eta=(t\cos \phi, t\sin \phi,\eta_3,...,\eta_m), t\ge 0, \phi\in S_k\},\- {\rm{and}}\]
\[F_k=\{\eta\in\Sm, \eta=(t\cos \phi, t\sin \phi,\eta_3,...,\eta_m), t\ge 3/4, \phi\in T_k\}.\]
Let 
$f_k=1$ on $E_k$ and $f_k=-1$ on $\Sm\setminus E_k$. Further, let $u_k=P\ast f_k$ be the corresponding harmonic function in the unit ball $\Bm$.

\begin{lemma}
\label{l:ex1}
The function $u_k$ has the following properties
\begin{enumerate}
\item[(a)] $|u_k|\le 1$ on $\Bm$;
\item[(b)] $u_k(r\eta)\ge 0$ when $\eta\in E_k$;
\item[(c)] For each $d\in \N$ there exists $c_{d,m}$ such that $|u_k(x)|\le c_{d,m}2^{-kd}(1-|x|)^{-d}$;
\item[(d)] There exists $a_m$ such that $u_k(r\eta)>1/4$ when $\eta\in F_k$, and $|x|>1-a_m2^{-k}$. 
\end{enumerate}
\end{lemma}

\begin{proof}
By the maximum principle 
(a) follows immediately. 
Note further that $f_k(x_1,x_2,...,x_m)=-f_k(x_1,-x_2,...,x_m)$, and thus
\[u_k(x)=\frac{1}{\gamma_{m -1}}\int_{\Sm}\frac{1-|x|^2}{|x-y|^{m}}f_k(y)dy\]
satisfies $u_k(x_1,x_2,...,x_m)=-u_k(x_1,-x_2,...,x_m)$. In particular,
\beq
\label{eq:ex1}
u_k(x_1,0,...,x_m)=0.
\eeq
 
Let $\alpha_k=\pi 2^{-k}$ and
\[A_k=\left[\begin{array}{ccccc}
\cos\alpha_k&-\sin\alpha_k&0&\cdots&0\\
\sin\alpha_k&\cos\alpha_k&0&\cdots&0\\
0&0\ &\ &\ &\\
\vdots&\vdots&\ &I_{m-2}&\ \\
0&0\ &\ &\ &\\
\end{array}\right],\]
where $I_{m -2} $ is the identity matrix.
Then $A_k$ is an orthogonal matrix and the corresponding transformation of $\R^m$ maps the unit sphere to itself, moreover $A_k(E_k)=\Sm\setminus E_k$. Then 
\[f_k(A_kx)=-f_k(x)\quad{\rm{and}}\quad
u_k((A_k)^{-1}x)=-u_k(x).\]
 Now, taking into account (\ref{eq:ex1}), we get
\[u_k(s\cos l\alpha_k, s\sin l\alpha_k,x_3,...,x_m)=0\]
for any $l=0,1,..., 2^{k+1}-1$.
Fix $l$ and consider the set
\[ G_{k,l}=\{x\in \Bm, x=(s\cos\phi, s\sin\phi,x_3,...,x_m), \phi\in(l\alpha_k,(l+1)\alpha_k)\}.\]
The boundary of $G_{k,l}$ consists of a part of the unit sphere and of subsets of the hyperplanes
\[\{(\sin l\alpha_k)x_1-(\cos l\alpha_k)x_2=0\}\quad\mathrm{ and } \quad\{(\sin (l+1)\alpha_k)x_1-(\cos (l+1)\alpha_k) x_2=0\}.\]
 On both subsets of the hyperplanes $u_k=0$, and on the corresponding part of the sphere  all boundary values of $u_k$ equal $1$ if $l$ is even and $-1$ if $l$ is odd. Anyway, $u_k$ does not change sign in $G_{k,l}$ and (b) follows.

To prove (c) assume first that $d=1$. We write
\begin{multline*}
u_k(x)=\frac{1}{\gamma_{m -1}}\int_{E_k}\frac{1-|x|^2}{|x-y|^m}dy-\frac{1}{\gamma_{m -1}}
\int_{\Sm\setminus E_k}\frac{1-|x|^2}{|x-y|^m}dy=\\
\frac{1}{\gamma_{m -1}}\int_{E_k}\left(\frac{1-|x|^2}{|x-y|^m}-\frac{1-|x|^2}{|x-A_ky|^m}\right)dy.
\end{multline*}
We want to estimate the difference under the integral sign. Note 
that \[\max_{y\in\Bm}|y-A_ky|=2\sin \alpha_k/2<\alpha_k\] and assume that $1-|x|>\alpha_k$, then
\[
\left|\frac{1}{|x-y|^m}-\frac{1}{|x-A_ky|^m}\right|\le \frac{m|y-A_ky|(|x-A_ky|+\alpha_k)^{m-1}}{|x-y|^{m}|x-A_ky|^m}\le \frac{m\alpha_k 2^{m-1}}{(1-|x|)|x-y|^m}.\]
We obtain $|u_k(x)|\le c_m\pi 2^{-k}(1-|x|)^{-1}$ when $1-|x|>\alpha_k$, otherwise the inequality follows from (a).

In general, we  write
\[
u_k(x)=\frac{1-|x|^2}{\gamma_{m -1}}2^{-d +1}\int_{E_k}\sum_{j=0}^d(-1)^j{d \choose j}f_x(A_k^jy)dy,
\]
where $f_x(y)=|x-y|^{-m}$. To estimate the sum under the integral sign let \[y=(y_0\cos\psi,y_0\sin\psi,y_1)\in\R\times\R\times\R^{m-2}.\]
We have $A_k^jy=(y_0\cos(\psi+j\alpha_k),y_0\sin(\psi+j\alpha_k),y_1)$ and $f_x(y)=h_{x,y_0,y_1}(\psi)$.
 Then  we write the Taylor polynomial of order $d-1$ and estimate the $d$th derivative of $h_{x,y_0,y_1}$. Finally, applying the difference relation
 \[
 \sum_{j=0}^d(-1)^j{d \choose j} j^l=0,\]
 when $l <d$ (see for example \cite [p.\,42] {V}), we get that
\[
\left|\sum_{j=0}^d(-1)^j{d \choose j}f_x(A_k^jy)\right|\le c_{d,m}2^{-dk}(|x-y|-d\alpha_k)^{-m-d}\le c_{d,m}2^{-dk} |x -y |^{-m} (1 - |x |)^{-d}.\]
Then for  $1-|x|>2d\alpha_k$, we have
\[
u_k(x)\le c_{d,m}2^{-dk}(1-|x|)^{-d},\]
and the same can be shown for $1-|x|<2d\alpha_k$ since $|u_k |\le 1 $, so (c) follows.

Finally, we prove (d). Let  $\eta\in F_k$ and $x=r\eta$. It is easy to check that $B(\eta, 2^{-k-1})\subset E_k$. A direct calculation shows that for $a_m$ small enough
\[\frac{1}{\gamma_{m -1}}\int_{\Sm\setminus B(\eta, 2^{-k-1})}\frac{1-r^2}{|r\eta-y|^m}dy<\frac 3 8,\]
when $r>1-a_m2^{-k}$.  Thus $u_k (x)>1 -2\fraction38 =\fraction14 $.
\end{proof}
It will be more convenient to use functions like $\Re(z^{2^k})$, so we define 
\[
h_k(x)=u_k((A_{k+1})x);\quad B_k=\cup_{j=0}^{2^k-1}[(2j-1/4)\pi2^{-k}, (2j+1/4)\pi2^{-k}].\] 
It is easy to check that (d) implies $h_k(r\eta)>1/4$ when 
\beq
\label{eq:last}
\eta\in H_k=\{y\in S, y=(t\cos\phi, t\sin\phi,y_3,...,y_m), t\ge 3/4, \phi\in B_k\}
\eeq 
and $r>1-a_m2^{-k}$.

\subsection{Construction of $u\in\Psi^m_v $ with a large set of radial growth}
\label{construction}

Now we can prove Theorem \ref{th:haus} (b).

\begin{proof}[Proof of Theorem \ref{th:haus} (b)]
First we construct $\nu_\beta(t)=O(tv(1-t)^\beta)$. 
The assumption on $v $ implies
$$\frac{d}2 v \left( 1 -\frac{d}2\right)^\alpha \leq \frac{d}2 D^\alpha v (1 -d)^\alpha <\fraction34 d v (1-d)^\alpha $$
 when $\alpha \le \alpha_0$. 
 
 For simplicity we define a new function $g $ such that $v(r)=g(\frac1{1-r})$. We will keep this notation throughout the paper. Then \eqref{vdouble} is equivalent to
\beq
\label{gdouble}
 g (2x) \le D g (x).  
 \eeq
  
  We choose $\alpha\le \min\{\alpha_0,  \beta\}$, and
 define  $\nu_\beta $ by 
 $$\nu_\beta (\pi2^{- n}) = \pi2^{-n} v (1 -2^{-n})^\alpha=\pi2^{-n} g(2^{n})^\alpha, \qquad n\ge 2 ,$$
  and $\nu_\beta $ is linear on $[\pi2^{-n-1},\pi2^{-n} ]$. 
   Then $\lim_{t \rightarrow 0} \nu_\beta (t) =0 $ and $\nu_\beta $ is continuous and increasing.   
For $t \in [\pi2^{-n -1},\pi2^{-n}	) $ we have
\begin{eqnarray*}
\frac{\nu_\beta (t)}{t g (\fraction1t)^\beta} &\le&  \frac{\pi2^{-n} g(2^{n})^\alpha}{\pi2^{-n-1}g (2^{n}/\pi)^\beta} = 2\frac{g (2^{n})^\alpha}{g (2^{n}/\pi)^\beta} 
\le 2\frac{g (2^{n})^\alpha}{g (2^{n-2})^\beta} \le 2D^{2\beta}\frac{g (2^{n})^\alpha}{g (2^{n})^\beta} \le 2D^{2\beta} 
\end{eqnarray*}
when $n\ge n_0 $, so $\nu_\beta (t)=O (t v (1 -t)^\beta) $ when $t \rightarrow 0 $. We also define a new function $\lambda_\beta (t) =t^{m -2}\nu_\beta (t)$.

Fix $A_1>1$ and define $b_1=1$,
\[
b_{n+1}=\min\{l: g(2^l)>A_1g(2^{b_n})\},\quad n=2,3,... .\]
By assumption, $g(2^{b_{n +1}}) \leq Dg (2^{b_{n +1} -1} )$, and by the way  the $b_n $'s are defined, 
$g (2^{b_{n +1} -1}) \leq A_1 g (2^{b_n}) $.  Then
\begin{equation}
\label{A_2}
g (2^{b_{n+1}}) \leq DA_1 g (2^{b_n})=A_2 g(2^{b_n}).
\end{equation}

Let
\[
u(x)=\sum_{n=1}^\infty g(2^{b_n})h_{b_n}(x).\]
We want to check that $u$ converges uniformly on compact subsets of $\Bm$ and $u\in\Phi^m_v$. 
Since $g $ fulfills \eqref{gdouble}, there exists $\gamma $  such that
\beq
\label{eq:ex2}
\frac{g(2^{l_2})}{g(2^{l_1})}\le 2^{\gamma(l_2-l_1)}
\eeq
for  $l_1,l_2\in\N $, just let $\gamma =\log_2 D $. 
Choose $d>\gamma $ and note that (\ref{eq:ex2}) implies
\beq
\label{eq:ex3}
\frac {g(2^{b_{n+1}}) 2^{-b_{n+1}d}}{g(2^{b_n})2^{-b_n d}}\le 2^{-(d-\gamma)(b_{n+1}-b_n)}\le 2^{-(d-\gamma)},\eeq
when $n>n_0$.
Assume that $1-2^{-b_N}<|x|<1-2^{-b_{N+1}}$, then by Lemma \ref{l:ex1} (a) and (c),
\[
|u(x)|\le \sum_{n=1}^Ng(2^{b_n})+c_{d,m}\sum_{n=N+1}^\infty g(2^{b_n})2^{-b_nd}(1-|x|)^{-d}.\]
The first sum is bounded by $C_1g(2^{b_N})$, and for $N$ large enough (\ref{eq:ex3}) implies that the second sum is bounded by $C_2g(2^{b_{N+1}})2^{-b_{N+1}d}(1-|x|)^{-d}\le C_2g(2^{b_{N+1}}).$
Then by (\ref{A_2})
\[
|u(x)|\le C_3 g(2^{b_{N+1}})\le C_4g(2^{b_N})\le C_4g\left(\frac1{1-|x|}\right).\]

Finally, we show that $F=\cap_{n} H_{b_n}\subset E^+(u)$, where $H_k$ are defined by (\ref{eq:last}). Let $x=|x|\eta,\, \eta\in F\subset\Sm$,
and $1-a_m2^{-b_N}<|x|\le 1-a_m2^{-b_{N+1}}$, where $a_m$ is as in Lemma \ref{l:ex1}; we may assume also that $x_1^2+x_2^2>1/4$. Then by  (b) and (d) in Lemma \ref{l:ex1} (see also the definition of $h_k$ above), we obtain
\[
u(x)=\sum_{n=1}^{\infty}g(2^{b_n})h_{b_n}(x)\ge g(2^{b_N})h_{b_N}(x)\ge\frac{1}{4} g(2^{b_N})\ge C_5g\left(\frac1{1-|x|}\right).\]

Let $C=\cap_{n}B_{b_n}\subset [0,2\pi)$ and $C_j=\cap_{n=1}^j B_{b_n}$.  
Then $C_j$ is  a  union of $N_j$ intervals of length $l_j=\frac{\pi}{4}2^{-b_j}$, where some of the intervals are next to each other, and $C $ is a set as in Lemma A. 
Intervals of length $l_j $ are called 
intervals from $j$-th generation. Each 
  of them contains $k_{j+1}$ intervals from the next generation.  It is easy to show that  $k_{j+1} =1 $ if $b_{j+1} -b_{j} =1 $, and  $k_{j +1} = \frac14 2^{b_{j+1}-b_j}$ if $b_{j+1} -b_{j} >1 $.  So $k_{j +1} \ge \frac14 2^{b_{j+1}-b_j}$ and 
  $N_j \ge (\frac{1}4)^j 2^{b_{j}}$.  
  
Let $0 <l \le \frac{\pi}{4} $ and pick $t $ and $j $ in $\N $ such that $l_j \ge \frac{\pi}{4}2^{-t}\ge l \ge \frac{\pi}{4}2^{-t -1}\ge l_{j+1}$. Then 
\begin{eqnarray*}
\frac{\nu_\beta(l)}{l}&\ge& \frac{\nu_\beta (\frac{\pi}{4}2^{-t-1})}{\frac{\pi}{4}2^{-t}}= \frac{\pi 2^{-t -3}g (2^{t +3})^\alpha}{\pi2^{-t -2}} 
\ge \fraction12g (2^{b_j+3})^\alpha 
\ge \fraction12g (2^{b_j})^\alpha \\
&\ge& \frac{1}{2A_2^\alpha}g (2^{b_{j+1}})^\alpha
\ge \frac{1}{2A_2^\alpha D^{2\alpha}}g \left(2^{b_{j+1}}4\right)^\alpha = \frac{1}{2A_2^\alpha D^{2\alpha} } \frac{\nu_\beta(l_{j+1})}{l_{j+1}}.
\end{eqnarray*}
Lemma A 
  with $\nu_\beta $ defined as above and $a = \frac{1}{2A_2^\alpha D^{2\alpha}} $ 
   now yields
\begin{eqnarray*}
\HH_{\nu_\beta}(C)&\ge&  \frac{a}{2}\liminf_{j\rightarrow\infty} N_j\nu_\beta(l_j) 
\ge \frac{a}{2}\liminf_{j\rightarrow\infty}  \left(\frac{1}4\right)^j 2^{b_{j}}\nu_\beta \left(\frac{\pi}{4}2^{-b_j}\right) \\
&= &\frac{a\pi}{8}\liminf_{j\rightarrow\infty}\left(\frac{1}4\right)^j g (2^{b_j +2})^\alpha 
\ge \frac{a\pi}{8}\liminf_{j\rightarrow\infty} \left(\frac{1}4\right)^j A_1^{j\alpha} g (2^{b_0})^\alpha.
\end{eqnarray*}
By choosing $A_1^\alpha>4$ we obtain $\HH_{\nu_\beta}(C)=\infty$.

Then by Lemma \ref{crossproduct} for $\lambda_\beta(t)=t^{m-2}\nu_\beta(t)$ and the remark on the behavior of the Hausdorff measure under the Lipschitz map, we have \[\HH_{\lambda_\beta}(E^+(u))\ge\HH_{\lambda_\beta}(F)>0.\] 
\end{proof}

\section{Positive harmonic functions}

\subsection{Proof of the Theorem \ref{th:pi} (a)} 

We will now consider extremal growth 
 on subsets of radii of the unit ball in $\R^m $ for positive functions. 
Let $v$ be a positive increasing continuous function on $[0,1) $ and assume $\lambda(t)=t^{m-1}v(1-t)$ is increasing.  Let $u $ be a positive harmonic function on $\Bm $ and let $F^+_v (u)$ be defined by (\ref{eq:F}).
For positive $u\in\Psi_v^m $, clearly $E^+ (u)\subset F^+_v (u) $. 
Theorem \ref{th:pi} is a generalization of Theorem 2 in \cite{BLMT}, where the result is proved for $v (r) =\log (\fraction1{1 -r})$ and $m =2 $.  
Note that we do no longer assume that $u\in\Psi_v^m $.


The proof of Theorem \ref{th:pi} (a) is similar to the one in 
\cite{BLMT}, but the proof of Lemma \ref{lem:1} is new. 

Let 
\[
F_{n}=\left\{          \z\in S: \limsup_{r\rightarrow 1}\frac{u(r\z)}{v(r)}\ge\frac 2 n
                 \right\}.
\]
It suffices to prove  that
$\HH_\lambda (F_n ) <\infty$ for all   $n$. 

Clearly $u=P*\mu$ for some positive Borel measure $\mu$ on $S$. Let $h:S\rightarrow [0,\pi] $ be given by $h (\cos\p,\z') =\p $ and define a measure on $[0,\pi] $ by $\nu=h_*\mu $, which means that $\nu(A)=\mu(h^{-1}(A))$ for any measurable set $A \subset [0,\pi] $. 
The formula $$\int_0^\pi f ( \psi)d\nu(\psi) =\int_0^\pi f' (\p)\nu ((\p,\pi])
d\p $$  is valid for $f\in C^1 [0,\pi] $ that is non-decreasing and fulfills $f (0) =0 $ and $f (t)>0 $ for $t>0 $ (see for example \cite[p.\,84]{St}). By using it with $f (\p) =\tilde{P}_{m,r} (0)-\tilde{P}_{m,r} (\p) $, we get the following integration by parts on $S$
\begin{gather*}
\int_\Sm P (rx,\z)d\mu(\z) =\int_\Sm \tilde{P}_{m,r} (h (\z))d\mu(\z) =\int_0^\pi \tilde{P}_{m,r} (\p) d\nu \\
=- \int_0^\pi \left (\tilde{P}_{m,r} (0) -\tilde{P}_{m,r}(\p)\right)  d\nu +\tilde{P}_{m,r} (0)\nu ([0,\pi])\\
=-\int_0^\pi Q_m (r,\p) \nu ((\p,\pi])d\p +\tilde{P}_{m,r} (0)\mu (\Sm) \\
=\tilde{P}_{m,r} (\pi)\mu (\Sm) +\int_0^\pi Q_m (r,\pi)\nu ((0,\p])d\p ,
\end{gather*}
thus
\beq
\label{pq}
\int_\Sm P (rx,\z)d\mu(\z) =\tilde{P}_{m,r} (\pi)\mu (S) +\int_{0}^\pi Q_m(r,\phi)\mu(\bar {B} (x,\phi))d\phi .
\eeq

We need the following lemma:

\begin{lemma}
\label{lem:1}
For each $n$  there exists  $k =k (m,n)>0 $ such that for any $x\in F_n$ there is a decreasing sequence
$\{\Delta_j\} $, $\Delta_j\rightarrow 0$ as $j\rightarrow \infty$, which  satisfies
\beq
\label{eq:20}
\mu(B (x,\Delta_j) )  \ge   k \sigma (B (x,\Delta_j) ) v (1 -\Delta_j).
\eeq
 \end{lemma}

Suppose this lemma is already proved. Let $K$ be a compact subset of $F_n$ and let $B_j =B (x_j,a_j) $, where $x_j\in K $ and $a_j <\epsilon $. 
For each $\epsilon >0$ we can cover $K$ with a finite collection of such balls $\{  B_j:{j\in J}\} $ 
which satisfy
 $\mu(B_j )\ge k\sigma (B_j ) v (1 -a_j)$. 
   By the Vitali covering lemma (see for example \cite[p.\,2]{He}) there exists a subcollection $J' \subseteq J$ where $\{B_j:j \in J'\} $ are disjoint and 
$\cup_{j \in J} B_j \subseteq \cup_{j \in J'} 5B_j$.
Using \eqref{cap}
and Lemma \ref{lem:1} we obtain
\begin{eqnarray*}
\sum_{j\in J'} \lambda(5a_j)
&\le& C\sum_{j\in J'}\sigma (B(x_j,5a_j)) v (1 -a_j)
\le C 5^{m -1}\sum_{j\in J'}\sigma (B_j ) v (1 -a_j)\\
&\le& C\frac {5^{m -1}} {k}\sum_j\mu(B_j
)\le C\frac {5^{m -1}} {k} \mu(S),
\end{eqnarray*}
which yields $ {\HH}_\lambda(K)\le C\frac {5^{m -1}} {k}\mu(S)$. Thus ${\HH}_\lambda(F_n)\le C\frac {5^{m -1}} {k}\mu(S)<\infty$.
 
\begin{proof}[Proof of Lemma \ref{lem:1}] Assume that $x = (1,0,...,0).  $
Then by \eqref{pq} and \eqref{q2},
\bas
u (rx) &=& 
   \int_{S} P(rx,\z)d\mu (\z)\le
      \mu(S)  +   \int_{0}^\pi \mu(B (x,\phi)) Q_m(r,\phi)d\phi   \\
 &\le& \mu(S)  +   
 \int_{0}^{d} \mu(B (x,\phi)) 
   Q_m(r,\phi)d\phi   +   \mu(S) \int_{d}^\pi  Q_m(r,\phi)d\phi  \\
 &\le& \mu(S)  + \int_{0}^{d} \mu(B (x,\phi))Q_m(r,\phi)d\phi +\mu(S)C_6 
  d^{-m}   
\eas
Let $k  < [(C_2 C_4 +C_9) 3n]^{-1}$, where the constants are from \eqref{cap}, \eqref{q1} and \eqref{q3}. 
For $x \in F_{n}$  there exists a sequence $\{r_j\}_1^\infty$ such that
  $r_j\nearrow 1$ and $u (r_j)>\fraction 1n v(r_j) $. We may assume that $r_j>\fraction12 $. Now choose $d_j =d_j(r_j,v) \ge 1 - r_j$ such that 
  $$C 
  d_j^{-m} <\fraction1{3n} v (r_j) $$ and $d_j\rightarrow 0 $ when $j\rightarrow \infty $.
 For $ j>j_0$, $\mu(S)<\fraction  1{3n} v(r_j) $. 
 Then 
 $$\int_{0}^{d_j} \mu(B (x,\phi))Q_m(r_j,\phi)d\phi>\fraction1 {3n} v (r_j). $$ We claim that this implies  that for any $j $ there exists $\Delta_j\in (0,d_j) $ such that 
 $$\mu(B (x,\Delta_j) )  \ge   k \sigma (B (x,\Delta_j) ) v (1 -\Delta_j) , $$
 and the lemma follows. 
 If not,  
 there exists $j $ such that 
 $$\mu(B (x,\p) )  <   k \sigma (B (x,\p) )v (1 -\p) $$ for any $\p\in (0,d_j)$. Using \eqref{cap} and the fact that $t^{m - 1} v (1 -t) $ is increasing, and then applying 
  \eqref{q1} and \eqref{q3}, we obtain
\begin{gather*}
\int_{0}^{d_j} \mu(B (x,\phi))Q_m(r_j,\phi)d\phi 
<k \int_{0}^{d_j} \sigma(B (x,\phi)) v (1 -\p) Q_m(r_j,\phi)d\phi \\
\le k C_2 (1 -r_j)^{m -1} v (r_j) \int_{0}^{1 -r_j}  Q_m(r_j,\phi)d\phi +k  v (r_j) \int_{1 -r_j}^{\pi} \sigma(B (x,\phi)) Q_m(r_j,\phi)d\phi \\
\le k\left ( C_2  (1 -r_j)^{m -1} v (r_j) C_4 \fraction1 {(1 -r_j)^{m -1}} + C_9 v (r_j)\right) < \fraction1 {3n}  v (r_j),
\end{gather*} 
and we have a contradiction.
\end{proof}

\begin{corollary} 
If $v(r) =(\fraction1{1 -r})^{\gamma} $ for $0\le\gamma< m -1 $, then for any positive harmonic function $u$, 
the Hausdorff dimension of $F_v^+ (u) $ is less than or equal to $ m -1 -\gamma$. 
\end{corollary} 
We will show in section \ref{positive sharpness section} that this estimate is sharp, i.e.  there exists $u $ such that $$\dim F_v^+ (u) =m -1 -\gamma .  $$
\subsection{Measures that correspond to positive functions in $\Psi_v^m $} 
Let
\[
 \Theta^m_v=\{u:\Bm\rightarrow\R, \Delta u=0, 0 <u(x)\le Kv(|x|)\}.\]
 We want to characterize all functions in $\Theta^m_v$ by their corresponding measure on $\Sm $.
 
\begin{proposition}
\label{mu}
Suppose $v $ satisfies \eqref{vdouble} and let
$u (x) =\int_S P (x,\z)d\mu (\z) $ where $\mu $ is a positive Borel measure on $\Sm $. Then $u\in \Theta^m_v $ if and only if
\beq
\label{eq:mu}
\mu(B (x,\p))\le C\sigma (B (x,\p))g \left(\frac{\pi}{\p}\right)
\eeq
for each ball   $B (x,\p)\subset \Sm$.\end{proposition}

\begin{proof}
Assume $u\in\Phi^m_v $ and let $x\in S $. 
Then
\begin{eqnarray*}
v \left(1 -\fraction{\p}{\pi}\right)&\ge& \fraction1K u \left(\left(1 -\fraction{\p}{\pi}\right)x\right)  =   \frac1{K\gamma_{m -1}}\int_{S}\frac{1 - (1 -\fraction{\p}{\pi})^2}{|x (1 -\fraction{\p}{\pi})-\z |^m} d\mu (\z)\\
&\ge& \frac1{K\pi\gamma_{m -1}}\int_{B (x,\p)} \fraction{2\p - \p^2}{(2\p  )^{m}} d\mu (\z) \ge\frac1{K\pi\gamma_{m -1}}\fraction{\mu (B (x,\p))}{2^m \p^{m-1}}\\
&\ge& C \mu (B (x,\p)) \fraction1{\sigma (B (x,\p) )},
\end{eqnarray*}
thus $\mu(B (x,\p))\le C\sigma (B (x,\p))g \left(\frac{\pi}{\p}\right) $.

Conversely, suppose that (\ref{eq:mu}) is fulfilled.  Assume for simplicity that $x = (1,0,...,0) $. 
Then by \eqref{pq},
\begin{gather*} 
u (rx) =
   \int_{S} P(rx,\z)d\mu (\z)\le
      \mu(S)  +   \int_{0}^\pi \mu(B (x,\phi)) Q_m(r,\phi)d\phi   \\
    \le \mu(S)  +   \int_{0}^{1 -r} \mu(B (x,\phi)) Q_m(r,\phi)d\phi   +   \int_{1 - r}^\pi \mu(B (x,\phi)) Q_m(r,\phi)d\phi  \\
 \le \mu(S)  +   \mu(B (x,1 - r)) \int_{0}^{1 - r} 
   Q_m(r,\phi)d\phi  +   C\int_{1 - r}^\pi \sigma (B (x,\phi))g \left(\fraction\pi\phi \right) Q_m(r,\phi)d\phi  \\
 \le \mu(S)  +  C\sigma (B (x,1 - r))g \left(\fraction\pi{1 - r}\right)  \int_{0}^{1 - r} Q_m(r,\phi)d\phi  \\
 +  \, Cg \left(\fraction\pi{1 - r}\right) \int_{1 - r}^\pi \sigma (B (x,\phi))Q_m(r,\phi)d\phi.  
\end{gather*} 
Furthermore, by \eqref{q1} and \eqref{q3}, $u (rx)\le \tilde {C} g (\fraction\pi{1 - r}) \le\tilde {C}_1 v (r)$.

\end{proof}

 
\subsection{Proof of Theorem \ref{th:pi} (b)}\label{positive sharpness section}


First, note that $v $ satisfies \eqref{vdouble}, in fact 
\[v\left(1 -\fraction{t}2\right) = \fraction{2^{m -1}\lambda \left(\fraction {t}2\right)} {t^{m -1}}< \fraction{2^{m -1}\lambda (t)}{t^{m -1}}=2^{m -1} v (1 -t).\]

Consider the set $A =[0,\pi]\times ...\times [0,\pi]\times [0,2\pi]$ and the hyperspherical coordinates on $S$, i.e., we consider the function $f: A\rightarrow S$ defined by 
$$f (\phi_1,..,\phi_{m-1}) = (\cos \phi_1,\sin \phi_1 \cos \phi_2,..., 
\sin \phi_1 ...\sin \phi_{m -1}). $$
This function is bilipschitz on $[1,2]^{m -1} $. We will use a Cantor-type construction to get a set $C\subset [1,2]^{m -1}\subset A $. We first construct a set in $[1,2] $. 

Let $F_0 =[1,2] $. Define $d_k $  as 
\[d_{k}=\min \{n \in \N: g(2^n)\ge 2^{(m -1)k}\},\quad k=2,3,...\]
Then $d_{k+1}>d_k $ for all $k $. By (\ref{gdouble}) we also have
\beq
\label{d}
\fraction{g (2^{d_k})} {g (2^{d_{k -1}})}\le \fraction{Dg (2^{d_k -1})} {g (2^{d_{k -1}})}\le \fraction{D2^{(m -1)k}} {2^{(m -1) (k -1)}}=D2^{m -1} =\delta_m .
\eeq
We construct by induction sets $F_k\subset F_{k-1}$ such that $F_k$ consists of $n_k =2^{d_k-k }$ closed intervals of  
length $2^{-d_k}$ each. To obtain $F_k$ we divide each of the intervals of $F_{k-1}$ into $2^{ d_k -d_{k-1}}$ equal subintervals and choose each second 
of them for $F_k$. Now let $C_k = F_k\times ...\times F_k\subset [1,2]^{m -1} $.  The number of squares in $C_k $ is $N_k =2^{(d_k-k)(m -1) }$. Let also $C =\cap C_k $.

Let $\nu_k $ be the measures defined by $d\nu_k=2^{(m -1)k}\chi(C_k)dy$ on $[1,2]^{m -1} $, where $\chi(C_k)$  is the characteristic function of $C_k$. We also define the measures $\mu_k =f_*\nu_k$ 
on $S $. 
Denote 
$G_k =f (C_k) $ and 
 $G =f (C) $, clearly $G =\cap G_k $.

\begin{lemma}
\label{lem:4}
The sequence  $\{\mu_k\}$ converges $\ast $-weakly to a measure $\mu$ and $u=P*\mu\in\Theta^m_v$.  
\end{lemma}

\begin{proof}
The $\ast $-weak convergence of $\{\mu_k\} $ follows from the $\ast $-weak convergence of $\{\nu_k\} $, which we will prove now. 
Note that $\nu_k(C_0)=1$ for each $k$. Let $\{J_i\}_{i =1}^{N_k} $ be the squares of $C_k $. For each square $J_i $ 
the limit $\nu_k(J_i)$ as  $k\to \infty$ exists because
all values   $\nu_k(J_i)$  are the same  when $k>s$. For squares in $S\setminus C_k $ the limit will be $0 $. Now each continuous function on $[1,2]^{m -1}$ can be uniformly approximated by  linear combinations of characteristic functions of small squares. Thus for each continuous function $f$ on $[1,2]^{m -1} $ there exists
\[
\lim_{k\rightarrow\infty}\int_{[1,2]^{m -1}} fd\nu_k
\]
and $\nu_k$ converge weakly to some positive measure $\nu$. 

By Proposition \ref{mu} 
it suffices to check that
 \[
 \mu(B (x,r))\le C\sigma (B (x,r))g \left(\frac{\pi}{r}\right)
 \]
for each ball   $B (x,r)\subset \Sm$, in order to prove that $u= P*\mu \in \Theta^m_v$.
This is true if a similar estimate is true for $\nu $.

 Let $y\in [1,2]^{m -1}$  and 
 let $B_e(y,r) $ be a Euclidean ball. 
  Choose $s$ such that $2^{-d_{s}} <r\le  2^{-d_{s-1}} $.
 Now take a square 
 $Q \supset B_e (y,r)$ that is a union of dyadic cubes with side lengths $2^{-d_s} $, 
 and let the side length of $Q$ be $2^{-d_s}l $ for some $l\in \N $ such that $2^{-d_s}l <4r $.  
 Then $|Q |<A|B_e (y,r)|$ where $A=A (m) $. By using \eqref{d}, we obtain
 \bas
\nu_k(B_e (y,r))&\le& \nu_k(Q)=\nu_s(Q)\le 2^{(m -1)s}|Q|<A|B_e (y,r)|g\left(2^{d_s} \right) \\
&\le& A\delta_m |B_e (y,r)|g\left(2^{d_{s -1}} \right) 
\le A\delta_m |B_e (y,r)|g\left(\fraction1{r}\right),
\eas
which is the desired inequality. 
\end{proof}

To finish the proof of Theorem \ref{th:pi} (b) we will show that
$  G \subset E^+(v)$ 
and $\HH_\lambda(G)>0$. 
We have
 \[
u(rx)=\int_{S} P(rx,\z)d\mu(\z)\ge\int_0^\pi \mu(B (x,\phi)) Q_m(r,\phi)d\phi.\]
Let  $x\in \cap G_k=G$ and  $r\in [0,1) $. Choose $k_0 $ such that $2^{(m -1) (k_0-1)}\le g(\fraction1{1-r})< 2^{(m -1)k_0}$. 
Then  it is not difficult to see that
 $\mu(B (x,\phi))\ge c2^{(m -1)k_0}\phi^{m -1}$ for $\phi<1-r$ and
 \bas
u(rx) &\ge& \int_0^{1-r} \mu(B (x,\phi)) Q_m(rx,\phi)d\phi\\
 &\ge& c2^{(m -1)k_0}\int_0^{1-r}\phi^{m -1}  Q_m(rx,\phi)d\phi\ge c_12^{(m -1)k_0}\ge c_1 g \left(\fraction1 {1 -r}\right),
\eas
when $r>r_0$. Thus $G=\cap G_k\subset E^+(u)$.

Finally we use Lemma B to estimate $\HH_{\lambda}(C) $. We have
\begin{eqnarray*}
\liminf_{k\rightarrow\infty} N_k\lambda(l_k)
=\liminf_{k\rightarrow\infty} 2^{d_k -k}2^{- d_k} g \left( 2^{d_k}\right)
\ge \liminf_{k\rightarrow\infty}\pi 2^{-k+1} 2^{k}= c>0.
\end{eqnarray*}
Then $\HH_{\lambda}(C)>0 $ and $\HH_{\lambda}(G)\ge c_1\HH_{\lambda}(C)>0 $ since $C =f^{-1} (G) $ and $f^{-1} $ is Lipschitz. We use also that 
$$\lambda \left(\fraction{t}2\right) = \left(\fraction{t}2\right)^{m -1} v \left(1 -\fraction{t}2\right)\ge t^{m -1} v (1 -t)2^{-m +1} =2^{-m +1} \lambda (t), $$ see Section \ref{preliminaries haus}.


\bibliographystyle{plain}
\bibliography{generalcase3}

\begin{thebibliography}{10}

\bibitem{BM}
R.~Ba\~{n}uelos and Ch.~N. Moore.
\newblock {\em {Probabilistic Behaviour of Harmonic Functions.}}
\newblock {Birkh\"{a}user}, 1999.

\bibitem{BL}
A.~Borichev and Yu. Lyubarskii.
\newblock {Uniqueness theorems for Korenblum type spaces.}
\newblock {\em J. Anal. Math.}, 103:307--329, 2007.

\bibitem{BLMT}
A.~Borichev, Yu. Lyubarskii, E.~Malinnikova, and P.~Thomas.
\newblock {Radial growth of functions in Korenblum space}.
\newblock {\em Algebra i Analiz}, 21:47--65, 2009.

\bibitem{Ca}
L.~Carleson.
\newblock {On the existence of boundary values for harmonic functions in
  several variables.}
\newblock {\em Ark. Mat.}, 4:393--399, 1962.

\bibitem{D}
B.~E.~J. Dahlberg.
\newblock {Estimates of harmonic measure.}
\newblock {\em Arch. Ration. Mech. Anal.}, 65:275--288, 1977.

\bibitem{HA}
K.~Hatano.
\newblock {Evaluation of Hausdorff measures of generalized Cantor sets.}
\newblock {\em J. Sci. Hiroshima., Ser. A-I}, 1968.

\bibitem{He}
J.~Heinonen.
\newblock {\em {Lectures on analysis on metric spaces.}}
\newblock {Universitext. New York, Springer, 140 p. }, 2001.

\bibitem{HW}
R.~A. Hunt and R.~L. Wheeden.
\newblock {On the boundary values of harmonic functions.}
\newblock {\em Trans. Am. Math. Soc.}, 132:307--322, 1968.

\bibitem{JK}
D.~S. Jerison and C.~E. Kenig.
\newblock {Boundary value problems on Lipschitz domains.}
\newblock {Studies in Partial Differential Equations, MAA Stud. Math. 23,
  1-68}, 1982.

\bibitem{KK}
J.-P. Kahane and Y.~Katznelson.
\newblock {Sur le comportement radial des fonctions analytiques.}
\newblock {\em C. R. Acad. Sci. Paris S\'er. A-B}, 272:A718--A719, 1971.

\bibitem{K}
B.~Korenblum.
\newblock {An extension of the Nevanlinna theory.}
\newblock {\em Acta Math.}, 135:187--219, 1975.

\bibitem{LM}
Yu. Lyubarskii and E.~Malinnikova.
\newblock {Radial oscillation of harmonic functions in the Korenblum class}.
\newblock {arXiv:1006.2325v1}, 2010.

\bibitem{M}
P.~Mattila.
\newblock {\em {Geometry of sets and measures in Euclidean spaces. Fractals and
  rectifiability.}}
\newblock {Cambridge Studies in Advanced Mathematics. 44. Cambridge: Univ.
  Press., 343 p. }, 1995.

\bibitem{P}
I.~I. Privalov.
\newblock {\em {Boundary properties of analytic functions.}}
\newblock {Gosudarstv. Izdat. Tehn.-Teor. Lit., Moscow-Leningrad, 336 pp. },
  1950.

\bibitem{Se}
K.~Seip.
\newblock {An extension of the Blaschke condition.}
\newblock {\em J. Lond. Math. Soc., II. Ser.}, 51(3):545--558, 1995.

\bibitem{St}
D.~W. Stroock.
\newblock {\em {A concise introduction to the theory of integration. 3rd ed.}}
\newblock {Boston, MA: Birkh\"auser., 253 p. }, 1999.

\bibitem{V}
N.~Ya. Vilenkin.
\newblock {\em {Combinatorics.}}
\newblock {New York-London: Academic Press., 296 p.}, 1971.

\end{thebibliography}

\end{document}